\documentclass{publmathdeb}
\usepackage{amsfonts,amsmath,amsthm,amssymb, amscd}
\usepackage[top=2cm, bottom=5cm, left=3cm, right=1cm]{geometry}
\usepackage{enumerate}

\usepackage{xcolor}

\newcommand*{\gp}[1]{\langle\;#1\;\rangle}
\newcommand*{\Y}{\mathbin{\mathsf Y}}
\newcommand*{\order}[1]{\vert #1 \vert}

\newtheorem{theorem}{Theorem} 
\newtheorem{lemma}{Lemma}

\begin{document}

\title[The $*$-Unitary Isomorphism problem ]{The Isomorphism Problem of Unitary Subgroups of Modular Group Algebras}

\author[Balogh, Bovdi] {Zsolt Balogh, Victor Bovdi}

\dedicatory{Dedicated to the memory of Professor J\'anos Kurdics}

\address{Department of  Mathematical Sciences,	UAEU,  United Arab Emirates}
\email{\{baloghzsa,vbovdi\}@gmail.com}

\thanks{The research was supported by UAEU Research Start-up Grant No. G00002968}


\begin{abstract}

Let $V_*({F}G)$ be the normalized unitary subgroup of the  modular group algebra ${F}G$ of a finite $p$-group $G$ over a finite field ${F}$  with  the classical involution $*$.
We investigate the isomorphism problem for the group $V_*({F}G)$, i.e. we pose the question
 when  the  group algebra ${F}G$  is uniquely determined by  $V_*({F}G)$. We give affirmative answers  for classes of  finite abelian $p$-groups,  $2$-groups of maximal class   and  non-abelian $2$-groups of order at most $16$.
\end{abstract}

\maketitle
\section{Introduction and Results}

Let $V({F}G)$ be the  normalized unit group  of the  group algebra  ${F}G$ of a finite group $G$ over a  field ${F}$. In 1947 R.M.~Thrall proposed the following problem: {\it  for a given  group $G$ and a given field ${F}$, determine all groups $H$ such that ${F}H$ is isomorphic to ${F}G$ over ${F}$}. In the special case when $G$ is a $p$-group and ${{F}}$ is a field of characteristic $p$ this problem is called isomorphism problem of modular group algebras. The modular isomorphism problem has been investigated by several authors. It has long been known \cite{Deskins} that the group algebra ${F}G$ of a finite abelian $p$-group determines $G$. This result was extended to countable abelian $p$-groups by Berman \cite{Berman_I}. The most general result of the isomorphism  problem related to abelian  groups can be found in \cite{May_I}. For non-abelian group algebras, this problem was investigated in  \cite{Baginski_I, Baginski_II,  Baginski_Caranti,   Drensky, Salim_Sandling_III, Salim_Sandling_II, Salim_Sandling_I, Sandling_III, Sandling_IV} and \cite{Wursthorn_I}.  For an overview  we  recommend the survey paper \cite{Bovdi_survey}.

A modular group algebra has  a large group of units and the isomorphism problem can be generalized
in several ways. A  stronger variant of the isomorphism problem is  the isomorphism problem of normalized units \cite{Berman_I}. Let ${F}$ be a finite field of characteristic $p$, $G$ and $H$ finite $p$-groups such that $V({F}G)$ and $V({F}H)$ are isomorphic. Are $G$ and $H$  isomorphic?

Berman \cite{Berman_I}  gave a positive answer for his question for finite abelian $p$-groups.  Sandling \cite{Sandling_II} generalized Berman's  result,  proving  that if $G$ is a finite abelian $p$-group, then a subgroup of $V({F}G)$, independent as a subset of the vector space $FG$, is isomorphic to a subgroup of $G$. For finite non-abelian $p$-groups ($p>2$) with cyclic Frattini subgroup as well as for the class of  maximal $2$-groups  Berman's question  was solved    affirmatively \cite{Balogh_Bovdi_I, Balogh_Bovdi_II}.

An element $u \in V({F}G)$ is called {\it unitary} if $u^{-1}=u^*$ with respect to the classical involution of $G$ (which sends each element of $G$ into its inverse). Obviously, the set $V_*({F}G)$ of all
unitary elements of $V({F}G)$ forms a subgroup. The structure of the unitary subgroup $V_*({F}G)$ has been   investigated in \cite{Balogh_Laver, Bovdi_Szakacs_I, Bovdi_Grishkov,  Bovdi_Kovacs_I, Bovdi_Rosa_I, Creedon_Gildea_I, Creedon_Gildea_II, Gildea_I}.

The unitary group $V_*({F}G)$ of a group algebra ${F}G$ is a small subgroup in  $V({F}G)$. So it might be interesting to ask whether this smaller subgroup determines the basic group $G$. This problem is called the
{\it $*$-unitary isomorphism problem} (*-UIP) of group algebras.

Affirmative  answer  of (*-UIP) for the classes of finite abelian $p$-groups is listed here.

\begin{theorem}\label{T:1}
Let ${F}G$ be the group algebra of a finite abelian $p$-group $G$ over a  finite field ${F}$ of characteristic $p$. Let $V_*({F}G)$ be the unitary subgroup of the group of normalized units of ${F}G$ with  the classical involution $*$.
Then $V_*({F}G)\cong V_*({F}H)$ for some group  $H$ if and only if $G\cong H$.
\end{theorem}

Clearly, if $G\cong H$ then $V_*({F}G)\cong V_*({F}H)$. Therefore it is necessary to prove only the "if" part of the theorem.
For non-abelian $2$-groups we prove the following results:

\begin{theorem}\label{T:2}
Let $F$ be the field of $2$ elements,	and let $G$ and $H$ be finite $2$-groups of maximal class. Then 	 $V_*({F}G)\cong V_*({F}H)$ if and only if $G \cong H$.
\end{theorem}

\begin{theorem}\label{T:3}
Let ${F}G$ be the group algebra of a finite $2$-group $G$  of order at most $16$ over  the field ${F}$  of $2$ elements. If $H$ is a  $2$-group of order at most $16$, $V_*({F}G)\cong V_*({F}H)$ if and only if $G\cong H$.
\end{theorem}

\section{Lemmas and preliminaries}

Let $G$ be a finite abelian $p$-group. If $char(F)=p$, then (see \cite [Chapters 2-3, p.\,194-196]{Bovdi_survey})
\[
V(FG)=\Big\{\; x=\sum_{g\in G}\alpha_gg\in FG\; \mid\;  \chi(x)=\sum_{g\in G}\alpha_g=1\; \Big\},
\]
where $\chi(x)$ is the augmentation of the element $x\in FG$.

The subgroup of a group $G$ generated by all elements of order $p^i$ is denoted by $G[p^i]$.  The exponent of $G$ is denoted  by $\exp(G)$.  Set $G^{p^i}=\gp{ g^{p^i} \,\vert\,g \in G }$.  The number of subgroups of order $p^i$ in the decomposition of a finite abelian $p$-group $G$ into a direct product of cyclic groups is denoted by $f_i(G)$.

We use the following.
\begin{lemma}(\cite[Theorem 2]{Bovdi_Szakacs_I})\label{szakacs}
Let $G$ be a finite abelian $2$-group and $|{F}|=2^m\geq 2$. The group   $V_*({F}G)$ is a direct product of cyclic $2$-groups. 	Furthermore,
	\begin{enumerate}
		\item[(i)] the $2$-rank of $V_*({F}G)$ is equal to $\frac{m}{2}(|G|+|G[2]|+|G^2[2]|-|G^2|)-m$;
		\item[(ii)] $V_*({F}G)=G \times M$ and\\
		    $f_1(M)=t_0-2t_1+t_2-f_1(G)-f_2(G)+m(|G[2]|-1)$,\\
			$f_i(M)=t_{i-1}-2t_{i}+t_{i+1}-f_{i+1}(G)$,\\
where \;			$t_i=\frac{m}{2} (|G^{2^i}|-|G^{2^i}[2]|)$\quad and $i\geq 2$;
		\item[(iii)] $|V_*({F}G|=|G^2[2]|\cdot 2^{\frac{m}{2}(|G|+|G[2]|)-m}$.
	\end{enumerate}
\end{lemma}

Now we can prove the following.

\begin{lemma}\label{L:2}
Let $G$ be a finite abelian $p$-group and $|{F}|=p^m\geq p$ in which     $p$ is odd. Then
	$V_*({F}G)\cong V_*({F}H)$ for some group $H$ if and only if $G\cong H$.
\end{lemma}
\begin{proof} If  $\exp(V_*({F}G))=p^e\geq p$, then $\exp(G)= p^e$ by  \cite[Theorem 1(ii), p.\,24]{Bovdi_Szakacs_I}(ii) which yields  that
\[
f_i(V_*({F}G)) = \textstyle \frac{m}{2}\big(|G^{p^{i-1}}|-2|G^{p^i}|+|G^{p^{i+1}}|\big).
\]
Since $V_*({F}G)\cong V_*({F}H)$, the group  $H$ is also a finite abelian $p$-group because $G \leq  V_*({F}G)$. Thus  $f_i(V_*({F}G))=f_i(V_*({F}H))$ for all $i>1$ and  $\exp(H)=p^e$.
Moreover, it is easy to check that

\[
\begin{split}
\order{G}-2\order{G^{p}}+\order{G^{p^{2}}} =& \order{H}-2\order{H^{p}}+\order{H^{p^{2}}} \\
&\vdots\\
\order{G^{p^{i-1}}}-2\order{G^{p^i}}+\order{G^{p^{i+1}}} =& \order{H^{p^{i-1}}}-2\order{H^{p^i}}+\order{H^{p^{i+1}}} \\
&\vdots\\
\order{G^{p^{e-3}}}-2\order{G^{p^{e-2}}}+\order{G^{p^{e-1}}} =& \order{H^{p^{e-3}}}-2\order{H^{p^{e-2}}}+\order{H^{p^{e-1}}} \\
\order{G^{p^{e-2}}}-2\order{G^{p^{e-1}}} =& \order{H^{p^{e-2}}}-2\order{H^{p^{e-1}}} \\
\order{G^{p^{e-1}}} =& \order{H^{p^{e-1}}}.
\end{split}
\]	
A straightforward calculation gives $\order{G^{p^i}}=\order{H^{p^i}}$, so  $f_i(G)=f_i(H)$ for all $i\geq 1$. \end{proof}

The next result is a simple consequence of  \cite[Theorem 1]{Bovdi_Unitarity}.

\begin{lemma}\label{L:3}
Let $G$ be a finite abelian $2$-group and $|{F}|=2^m\geq 2$. If  $V_*({F}G)=V({F}G)$, then one of the following conditions holds:
	\begin{enumerate}
    \item[(i)] $G$ is an elementary abelian $2$-group;
    \item[(ii)] $G$ is a cyclic group of order $4$, and $|{F}|=2$.
	\end{enumerate}	
\end{lemma}

\begin{lemma}\label{L:4}
Let $G$ be a finite abelian $2$-group and $|{F}|=2^m\geq 2$. The order of $V_*({F}G)$ determines the order of $G$.
\end{lemma}
\begin{proof}
Suppose that  $|G|=2^n$. According to Lemma \ref{szakacs} $(iii)$ we have
\begin{equation}\label{E:1}
\order{V_*({F}G)}=\order{G^2[2]}\cdot \order{{F}}^{\frac{\order{G}+\order{G[2]}}{2}-1}.
\end{equation}

Therefore, our lemma is true for $n=1,2$ by Lemma \ref{L:3}(i-ii).
Assume that $\order{G[2]}=2^s \geq 2$.
Since
\[
\order{{F}}^{\frac{\order{G}}{2}} \leq
\order{G^2[2]}\cdot \order{{F}}^{2^{s-1}-1}\cdot \order{{F}}^{\frac{\order{G}}{2}}=
\order{G^2[2]}\cdot \order{{F}}^{\frac{\order{G}+\order{G[2]}}{2}-1}=
\order{V_*({F}G)},
\]
so\quad   $\order{{F}}^{\frac{\order{G}}{2}}\leq |V_*({F}G)|$ \quad and\quad
\[
\order{{F}}^{\frac{\order{G}}{2}-1} <\order{{F}}^{\frac{\order{G}}{2}} \leq \order{V_*({F}G)} \leq \order{{F}}^{\order{G}-1} = \order{V({F}G)}.
\]
\end{proof}
The previous lemma states that the order of $G$ is determined by the isomorphism class of the unitary subgroup $V_*(FG)$. A similar statement seems to be true for non-abelian group algebras.

We freely  use the following.

\begin{lemma}\label{L:5}
If $G$ is  a finite abelian $2$-group of $\exp(G)=2^e$ in which  $e\geq 2$, then \;   $G^{2^{e-1}}\cong G^{2^{e-1}}[2]$.
\end{lemma}

\section{Proofs}

\begin{proof}[Proof of Theorem \ref{T:1}]

According to Lemma \ref{L:2},  our  theorem is true for odd  $p$.

Let $G$ be a finite abelian $2$-group and let $V_*({F}G)\cong V_*({F}H)$ for a group  $H$, where ${F}$ is a finite field of $char(F)=2$. Since $G \leq  V_*({F}G)$, so $H$ is also a finite abelian group and $\order{G}=\order{H}$ by  Lemma $\ref{L:4}$. Our theorem follows immediately from Lemma \ref{szakacs} (iii) and Lemma \ref{L:3} for groups of order $4$.
	
First, let  $\exp(V_*({F}G))=2$.
 Exponents  of $G$ and $H$ are also equal to $2$, which confirms our theorem.

If  $\exp(V_*({F}G))=4$, then   $\exp(G)=\exp(H)=4$ and\quad  $\order{G}+\order{G[2]}=\order{H}+\order{H[2]}$ \quad by Lemma \ref{L:5} and  Lemma \ref{szakacs}(i). Hence  $G^2[2]\cong H^2[2]$ by Lemma \ref{szakacs}(iii), so $G^2\cong G^2[2]\cong H^2[2]\cong H^2$ by Lemma \ref{L:5}, so  $G\cong H$.

Now let $\exp(V_*({F}G))=2^e$ with  $e>2$. Obviously (see Lemma \ref{szakacs}(ii)),   $V_*({F}G)=G \times M$,
\[
\textstyle
t_{e}=\frac{m}{2} (\order{G^{2^{e}}}-\order{G^{2^{e}}[2]})=0\quad \text{and}\quad t_{e+1}=\textstyle\frac{m}{2} \big(\order{G^{2^{e+1}}}-\order{G^{2^{e+1}}[2]}\big)=0.
\]
 Moreover, $t_{e-1}=\frac{m}{2} (\order{G^{2^{e-1}}}-\order{G^{2^{e-1}}[2]})=0$ by Lemma \ref{L:5} and so
\[
f_e(M)=t_{e-1}-2t_{e}+t_{e+1}-f_{e+1}(G)=0,
\]
Consequently,  $\exp(M)<\exp(G)$ and  $f_{e}(V_*({F}G))=f_{e}(G)$. Using  Lemma \ref{szakacs}(ii), we compute that
\[
\begin{split}
f_1(V_*({F}G))&=t_0-2t_1+t_2-f_2(G)+m(\order{G[2]}-1), \\
f_{2}(V_*({F}G))&=t_{1}-2t_{2}+t_{3}-f_3(G)+f_2(G), \\
&\vdots\\
f_i(V_*({F}G))&=t_{i-1}-2t_{i}+t_{i+1}-f_{i+1}(G)+f_i(G), \\
&\vdots\\
f_{e-2}(V_*({F}G))&=t_{e-3}-2t_{e-2}-f_{e-1}(G)+f_{e-2}(G),  \\
f_{e-1}(V_*({F}G))&=t_{e-2}-2t_{e-1}+t_{e}=t_{e-2}-f_{e}(G)+f_{e-1}(G), \\
f_{e}(V_*({F}G))&=f_{e}(G).
\end{split}
\]	
It is easy to check that $\textstyle\sum_{i=1}^{e-2} f_i(V_*({F}G))$ is equal to
\[
\begin{split}
\textstyle
t_0&-t_1+m(\order{G[2]}-1)-t_{e-2}-f_{e-1}(G)\\
&=\textstyle\frac{m}{2} (\order{G}-\order{G[2]})-\frac{m}{2} (\order{G^{2}}-\order{G^{2}[2]})\\
&\qquad\qquad
+m(\order{G[2]}-1)-t_{e-2}-f_{e-1}(G)\\
&=\textstyle\frac{m}{2} \big(\order{G}+\order{G[2]}
-\order{G^{2}}+\order{G^{2}[2]}\big)-m-t_{e-2}-f_{e-1}(G).\\
\end{split}
\]
Since $f_{i}(V_*({F}G))=f_{i}(V_*({F}H))$ for all $i\geq 1$, from
Lemma \ref{szakacs}(i)  we obtain that
\[
t_{e-2}+f_{e-1}(G)=t_{e-2}'+f_{e-1}(H),
\]
 where
$t_{e-2}'=\frac{m}{2} (\order{H^{2^{e-2}}}-\order{H^{2^{e-2}}[2]})$.
Using the  facts that $f_{e}(G)=f_{e}(H)$ and
\[
t_{e-2}=\textstyle\frac{m}{2}\Big(2^{f_{e-1}(G)}(2^{f_{e}(G)}-1)2^{f_{e}(G)}\Big),
\]
we conclude that $f_{e-1}(G)=f_{e-1}(H)$ and $t_{e-2}=t_{e-2}'$.

Similarly, for each $1\leq s \leq e-3$  we obtain  that $\sum_{i=1}^{s} f_i(V_*({F}G))$ is equal to
\[
\begin{split}
t_0-t_1&+m(\order{G[2]}-1)-t_{s}+t_{s+1}-f_{s+1}(G)\\
=&\textstyle\frac{m}{2} (\order{G}+\order{G[2]}
-\order{G^{2}}+\order{G^{2}[2]})-m-t_{s}+t_{s+1}-f_{s+1}(G).
\end{split}
\]
Thus $f_{s}(G)=f_{s}(H)$ for all  $1\leq s \leq e$,  which complete the proof. \end{proof}

\section{Group algebras of $2$-groups of maximal class}

Let $G$ be a $2$-group of maximal class. It is well-known that $G$ is one of the following groups:
the dihedral $D_{2^{n+1}}$, the generalized quaternion $Q_{2^{n+1}}$, or  the semidihedral   group $D_{2^{n+1}}^-$,  respectively.  Set
\begin{equation}\label{E:2}
\begin{split}
D_{2^{n+1}}&=\gp{\; a,b  \;\;\vert \;\;
	a^{2^n}=1, b^2=1, (a,b)=a^{-2},\; n\geq 2\;};\\
Q_{2^{n+1}}&=\gp{\; a,b \;\; \vert  \;\;
	a^{2^n}=1, b^2=a^{2^{n-1}},  (a,b)=a^{-2},\; n\geq 2\;};\\
D_{2^{n+1}}^-&=\gp{\; a,b  \;\;\vert \;\; a^{2^n}=1, b^2=1,
	(a,b)=a^{-2+2^{n-1}},\; n\geq 3\;},\\
\end{split}
\end{equation}
where $(a,b):=a^{-1}b^{-1}ab$.

In the sequel  we  fix the cyclic subgroup $C=\gp{a\;|\;a^{2^n}=1}\cong C_{2^n}$  of  $G$ from the presentation \eqref{E:2}.  If  $|F|=2$, then  each $x \in FC$ can be written as $x=x_1+x_2a$ in which  $x_1,x_2\in V({F}C^2)$ and  $C^2:=\gp{a^2}$.
Let us fix the following automorphism of ${F}C$:
\[
x\;\mapsto \widetilde{x}=\widetilde{ x_1+x_2a }=x_1+x_2a^{1+2^{n-1}} \in {F}C.
\]
\noindent
We use freely that  $x\in V_*({F}C_{})[2]$ if and only if
$x^*=x$. Moreover (see Lemma \ref{szakacs}(iii)),
\begin{equation}\label{UggF}
\order{V_*({F}C_{})}=|C_{}^2[2]|\cdot 2^{\frac{1}{2}(|C_{}|+|C_{}[2]|)}=2^{2^{n-1}+1}.
\end{equation}

\noindent
For each  $0 \leq i< 2^n$,  we define the following set
\begin{equation}\label{FccDDK}
H_i=\{\;  \; h \in V({F}C_{}) \quad \vert\quad
hh^*(1+a)^i(1+a^{-1})^i\in {F}C_{}^2 \; \;\}.
\end{equation}

\begin{lemma}\label{L:6}(\cite[Lemma 8]{Balogh_Bovdi_II})
	The set $H_i$ has the following properties:
	\begin{enumerate}
		\item[(i)] if $i\geq 2^{n-1}$, then $H_i=V({F}C_{})$;
		\item[(ii)] if $i<2^{n-1}$ and $i$ is odd, then $H_i$ is empty;
		\item[(iii)] if $l<2^{n-2}$, then $H_{2l}\leq V({F}C_{})$ and   $|H_{2l}|=2^{3\cdot 2^{n-2}+l}$.
	\end{enumerate}
\end{lemma}
The set of the $*$-symmetric elements in $ V({F}C_{})$  is denoted by $S_*({F}C_{})$.
\begin{lemma}\label{simm}
	The group \quad   $S_*({F}C_{})[2]=S_*({F}C_{})\cap  V({F}C_{})[2]$ \quad has order    $2^{2^{n-2}+1}$.
\end{lemma}
\begin{proof}
If 	$x=\sum_{i=0}^{2^{n}-1}\alpha_{i}a^{i} \in {F}C_{}$,  then $x^2=\sum_{i=0}^{2^{n-1}-1}(\alpha_i+\alpha_{i+2^{n-1}})a^{2i}$\quad  and
\[
x^*=\textstyle\alpha_0+\alpha_{2^{n-1}}a^{2^{n-1}}+\sum_{i=1}^{2^{n-1}-2}(\alpha_i+\alpha_{-i\pmod{2^n}})a^{i}.
\]
It follows that  each $x\in  S_*({F}C_{})[2]$ can be written in the following form
\[
\begin{split}
x=\alpha_0+\alpha_{2^{n-2}}(a^{2^{n-2}}&+a^{2^{n-2}})+\alpha_{2^{n-1}}a^{2^{n-1}}\\
&+\textstyle\sum_{i=1}^{2^{n-2}-1}\alpha_i(a^{i}+a^{2i}+a^{-2i}+a^{-i}),
\end{split}
\]
so  the number of all units in $S_*({F}C_{})[2]$ is equal to $2^{2^{n-2}+1}$. \end{proof}
Using  \eqref{E:2} and \eqref{UggF},  we  compute  the number $\Theta_{G}(2)$  of involutions in $V_*({F}G)$ for cases  $G\in\{ D_{2^{n+1}}, Q_{2^{n+1}}\}$.

\begin{lemma}\label{L:8}
The number $\Theta_{D_{2^{n+1}}}(2)$   is equal to $2^{2^{n}+2}-2^{3\cdot 2^{n-2}+1}$, where  $n\geq 2$.
\end{lemma}
\begin{proof} 	
Let $x=x_1+x_2b\in V({F}D_{2^{n+1}})$ in which  $x_1,x_2 \in {F}C_{}$.
It is easy to check that  $x\in V_*({F}D_{2^{n+1}})[2]$ if and only if $(x_1+x_2b)^2=1$ and $x_1+x_2b=(x_1+x_2b)^*=x_1^*+x_2b$, if and only if
\begin{equation}\label{dihedral}
\begin{cases}
x_1^2=x_2x_2^*+1;\\
x_1=x_1^*.
\end{cases}
\end{equation}
We consider the following two cases related to values of $\chi(x_1)$  and  $\chi(x_2)$.

\noindent
Case 1. Let $\chi(x_1)=1$  and  $\chi(x_2)=0$. Obviously,  $x_2\in\{0, \gamma(1+a)^i\}$ for some $\gamma\in V(FC)$ and $0<i<2^{n}$, because  $x_2\not\in V(FC)$.

For $x_2=0$, the number of different  $x_1$ satisfying the system of equations  \eqref{dihedral} coincides with  $|S_*({F}C_{})[2]|$.

If  $x_2=\gamma(1+a)^i$, then we consider the following cases:

\noindent
\underline{Case 1.1.} Let $0<i<2^{n-1}$.  The number of such different $x_2$ is  	 $\order{H_{i}}/\order{A_i}$ by \eqref{FccDDK}, in which
\[
A_i=\{u\in V({F}C_{})\,|\, u(1+a)^i=(1+a)^i \}.
\]
Since $A_i=1+Ann_{FC}((1+a)^i)$, where $Ann_{FC}((1+a)^i)$ is the annihilator of $(1+a)^i$, we have   $\order{A_i}=2^i$ by \cite{Hill_I}.

If   $i$ is odd, then $H_{i}$ is empty by  Lemma \ref{L:6}(ii), so there is no  such $x$. Furthermore, for $x_2=\gamma(1+a)^i$ in which $i=2k$,   the number of different $x_2$ (see   Lemma \ref{L:6}(iii))  is equal to
\[
\textstyle\frac{\order{H_{2k}}}{\order{A_{2k}}}=\frac{2^{3\cdot 2^{n-2}+k}}{2^{2k}}=2^{3\cdot 2^{n-2}-k}.
\]
Now we want to deal with different number of $x_1$ with a fixed $x_2$. We have showed (see above) that $i$ is even.   If $x_1'+x_2b\in V_*({F}G)[2]$, such that  $\chi(x_1')=1$, 	then $(x_1^{-1}x_1')^2=x_1^{-2}(x_1')^2=1$ and $x_1'=(x_1')^*$, so $x_1'\in x_1 S_*({F}C_{})[2]$.	
Thus, if   either $0< i=2k<2^{n-1}$ or $x_2=0$, then     the number of different units satisfying \eqref{dihedral} is
\begin{equation}\label{LLnfy}
\begin{split}
\bigg( 1+\textstyle\sum_{k=1}^{2^{n-2}-1} \frac{\order{H_{2k}}}{\order{A_{2k}}} \bigg)\cdot \order{S_*({F}C_{})[2]}
=& 2^{2^{n-2}+1} + (2^{3\cdot 2^{n-2}}) (2^{2^{n-2}}+2^{2^{n-2}-1}+\cdots +2^2)\\
=& 2^{2^{n-2}+1} + 2^2(2^{3\cdot 2^{n-2}})(2^{2^{n-2}-1}-1).
\end{split}
\end{equation}
\noindent
\underline{Case 1.2}. Let  $2^{n-1}\leq i< 2^n$. Clearly,  $H_i=V({F}C_{})$ by Lemma \ref{L:6}(i) and  the number of different  $x=x_1+x_2b\in V_*({F}G)[2]$  is
\begin{equation}\label{LLCeeh}
\begin{split}
\textstyle\sum_{i=2^{n-1}}^{2^n-1}  \frac{\order{H_{i}}}{\order{A_{i}}} \cdot \order{S_*({F}C_{})[2]}=&\textstyle\sum_{i=2^{n-1}}^{2^n-1} \frac{2^{2^n-1}}{2^i} \cdot (2^{2^{n-2}+1})\\
=
&\textstyle\sum_{i=0}^{2^{n-1}-1} 2^i \cdot (2^{2^{n-2}+1}) = (2^{2^{n-1}}-1) \cdot (2^{2^{n-2}+1}).
\end{split}
\end{equation}
Finally, the sum of  \eqref{LLnfy}  and \eqref{LLCeeh} gives  that
there is exactly $2^{3\cdot 2^{n-2}}\cdot (2^{2^{n-2}+1}-2)$ units of such form.

\noindent
Case 2. Let  $\chi(x_1)=0$ and  $\chi(x_2)=1$. Clearly,  $x_2$ is a unit and  $x_2x_2^*=(1+x_1)^2$ by  \eqref{dihedral}, where $1+x_1$ is a $*$-symmetric unit  and $x_2x_2^*\in V({F}C_{}^2)$. For a fixed unit $x_2$, the set
\[
\{1+x_1 \in S_*({F}C_{})\,\vert\, (1+x_1)^2=x_2x_2^* \}
\]
is a right coset of $S_*({F}C_{})[2]$ in $S_*({F}C_{})$.	Therefore,  the number of different $x_1$ satisfying $(1+x_1)^2=x_2x_2^*$ for a fixed $x_2$ is $|S_*({F}C_{})[2]|$.

Keeping mind that    $x_2\in H_0$,     the number of units in $V_*({F}D_{2^{n+1}})[2]$ is equal to
\[
\order{H_0} \cdot \order{S_*({F}C_{})[2]}=2^{3\cdot 2^{n-2}+2^{n-2}+1}=2^{2^{n}+1}.
\]

Finally, adding the  result  of both cases (1 and 2),  the number of elements in  $V_*({F}D_{2^{n+1}})[2]$ is equal to
\[
\Theta_{D_{2^{n+1}}}(2)=2^{3\cdot 2^{n-2}}\cdot (2^{2^{n-2}+1}-2)+2^{2^{n}+1}=2^{2^{n}+2}-2^{3\cdot 2^{n-2}+1}.
\]

\end{proof}

\begin{lemma}\label{L:9}
	The number $\Theta_{Q_{2^{n+1}}}(2)$  is equal to $2^{3\cdot 2^{n-2}+1}$, where  $n\geq 2$.
\end{lemma}
\begin{proof}
Let $x=x_1+x_2b\in V({F}Q_{2^{n+1}})$, where $x_1,x_2 \in {F}C_{}$. It is easy to check that   $x\in V_*({F}Q_{2^{n+1}})[2]$ if and only if $(x_1+x_2b)^2=1$ and $x_1+x_2b=(x_1+x_2b)^*=x_1^*+x_2a^{2^{n-1}}b$, if and only if
\begin{equation}\label{quaternion}
\begin{cases}
x_1^2=x_2x_2^*+1;\\
x_1=x_1^*;\\
x_2=x_2a^{2^{n-1}}.
\end{cases}
\end{equation}
\noindent
Case 1. Let   $\chi(x_1)=0$ and  $\chi(x_2)=1$.	According to  (\ref{quaternion}) we have $x_2(1+a^{2^{n-1}})=0$. Since $x_2$ is a unit,  $1+a^{2^{n-1}}=0$ which is impossible.

\noindent
Case 2. Let $\chi(x_1)=1$ and  $\chi(x_2)=0$.	 Since $x_2$ is not a unit, $x_2\in\{0, \gamma(1+a)^i\}$ for some $\gamma\in V(FC)$ and  $i>0$.
According to \eqref{quaternion},  we have $x_2(1+a^{2^{n-1}})=0$ which holds  if and only if either $x_2=0$ or $2^{n-1}\leq i$. 	If $2^{n-1}\leq i$, then   \eqref{dihedral} and \eqref{quaternion} give us   the same set of solutions. Consequently,  the number $\Theta_{Q_{2^{n+1}}}(2)$  is equal to	
\[
\begin{split}
\Biggl( 1+\textstyle\sum_{i=2^{n-1}}^{2^n-1} \frac{\order{H_{i}}}{\order{A_{i}}} \Biggr) \cdot \order{S_*({F}C_{})[2]} = (2^{2^{n-1}}) \cdot (2^{2^{n-2}+1})=2^{3\cdot 2^{n-2}+1}.
\end{split}
\]
\end{proof}

\begin{proof}[Proof of Theorem \ref{T:2}] 	Let   $\Theta_{G}(2)$ be  the number of elements in $V_*(FG)[2]$. We show  that
\begin{equation}\label{E:4}
\Theta_{Q_{2^{n+1}}}(2) < \Theta_{D^-_{2^{n+1}}}(2) < \Theta_{D_{2^{n+1}}}(2).
\end{equation}
First of all, we  prove the left inequality in \eqref{E:4}.

Let $x=x_1+x_2b\in V({F}D_{2^{n+1}}^-)$ in which  $x_1,x_2 \in {F}C_{}$. It is easy to check that  $x\in V_*({F}D_{2^{n+1}}^-)[2]$ if and only if
$(x_1+x_2b)^2=1$ and $x_1+x_2b=(x_1+x_2b)^*=x_1^*+\widetilde{x_2}b$,  if and only if
\begin{equation}\label{semidihedral}
\begin{cases}
x_1^2=x_2x_2^*+1;\\
x_1=x_1^*;\\
x_2=\widetilde{x_2}.
\end{cases}
\end{equation}
\noindent
Case 1. Let   $\chi(x_1)=1$ and  $\chi(x_2)=0$. If $x_2\in\{\; 0, \; \gamma(1+a)^i\; \mid\; i\geq {2^{n-1}}\}$ for some $\gamma\in V({F}C)$,
then the set of solutions of \eqref{quaternion} and \eqref{semidihedral} are the same, so $\Theta_{Q_{2^{n+1}}}(2) \leq \Theta_{SD^-_{2^{n+1}}}(2)$.

\noindent
Case 2. Let  $\chi(x_1)=0$ and  $\chi(x_2)=1$. There are no  units in  $V_*({F}Q_{2^{n+1}})[2]$  with such condition  by the case 1 in the proof of Lemma \ref{L:9}. However,  $b\in V_*({F}D^-_{2^{n+1}})[2]$ , so  the left inequality of  \eqref{E:4} holds.	

Finally,  let us prove the right inequality in \eqref{E:4}. If $x_2=\widetilde{x_2}$, then  systems of equations \eqref{dihedral} and \eqref{semidihedral} have the same  set of solutions.

Let   $x_1+x_2b \in V_*({F}D_{2^{n+1}})[2]$ such that    $x_2\not =\widetilde{x_2}$.  Note that there is no any unit  $x_1+x_2b \in V_*({F}D^-_{2^{n+1}})[2]$ by \eqref{semidihedral}. Set $s_1:=1+a^{2^{n-1}}$ and $s_2:=a\in FC_{}$.  It is easy to check that  $ s_1+s_2b \in V_*({F}D_{2^{n+1}})[2]$, such that $\chi(s_1)=0$,  $\chi(s_2)=1$ and $s_2\not=\widetilde{s_2}$. Consequently, $\Theta_{D^-_{2^{n+1}}}(2) < \Theta_{D_{2^{n+1}}}(2)$ and  the proof is complete.
\end{proof}


\begin{proof}[Proof of Theorem \ref{T:3}]
If  $G\in \{Q_8, D_8\}$, then  $V_*({F}G)\cong G\times C_2^3$  by \cite{Bovdi_Erdei_I, Creedon_Gildea_I, Creedon_Gildea_II}. Furthermore, \cite[Corollary 10]{Balogh_Creedon_Gildea} shows that $V_*({F}G)$ is Hamiltonian
if and only if $G$ is Hamiltonian, so this part is done.

\noindent	
Let $\frak{S}:=\{Q_8\times C_2, M_{16}, Q_{16}, C_4 \ltimes C_4, D^-_{16}, D_8 \Y C_4, D_{16}, G(4,4), D_8 \times C_2\}$ be the list of  non-abelian group of order $16$. A generator set of unitary subgroups $V_*({F}G)$  for  groups $G\in \frak{S}$ was given  in \cite{Bovdi_Erdei_II}. Using  this list,  we consider each group $G\in \frak{S}$ separately:

\noindent
$\bullet$ Let $G=Q_8 \times C_2$. The group $V_*({F}G)$ is Hamiltonian \cite[Corollary 10]{Balogh_Creedon_Gildea}, so  $V_*({F}G)\cong G \times C_2^7 \cong Q_8 \times C_2^8$ by \cite{Bovdi_Erdei_II}.

\noindent
$\bullet$ Let $G\in\{M_{16}, Q_{16}\}$, where  $M_{16}:=\gp{a,b\vert a^8=b^2=1,(a,b)=a^4}$.
Clearly,  $|V_*({F}G)|=2^{10}$  by \cite[Examples 3, 10]{Bovdi_Erdei_II}.

\noindent
According to \cite[Examples 3]{Bovdi_Erdei_II}, we have
\[
V_*({F}M_{16})\cong (G \ltimes C_2^3) \times C_2^3=(\gp{a,b}\ltimes \gp{c_1,c_2,c_3}) \times C_2^3,
\]
in which   $(c_1,a)=c_3$, $(c_2,a)=(c_3,a)=1$,  $(c_1,b)=c_2c_3$ and $(c_2,b)=(c_3,b)=1$.
It yields that   $V_*({F}M_{16})'\cong G'\times C_2 \times C_2 \cong C_2^3$.

\noindent
According to \cite[Examples 10]{Bovdi_Erdei_II}, we have
\[
V_*({F}Q_{16})\cong (G \ltimes C_2^4) \times C_2^2=(\gp{a,b}\ltimes \gp{c_1,\ldots, c_4})\times C_2^2,
\]
in which $(c_1,a)=(c_2,a)=c_1c_2$, $(c_3,a)=(c_4,a)=c_3c_4$,  $(c_1,b)=(c_2,b)=1$ and $(c_3,b)=(c_4,b)=c_3c_4$.
This yields that  $V_*({F}Q_{16})'\cong C_4 \times C_2^2$, so  $V_*({F}M_{16})$
is not isomorphic to $V_*({F}Q_{16})$.


\noindent
$\bullet$ Let $G\in \{C_4 \ltimes C_4, D^-_{16}, D_8 \Y C_4\}$.   Then  $|V_*({F}G)|=2^{11}$ (see \cite{Bovdi_Erdei_II, Gildea_I}) and using the generating set from \cite{Bovdi_Erdei_II} we prove that these unitary subgroups are pairwise  not isomorphic.


\noindent
Let  $G=C_4 \ltimes C_4$. According to \cite[Example 9]{Bovdi_Erdei_II},  we get that
\[
V_*({F}G)\cong (G \ltimes C_2^4) \times C_2^3=(\gp{a,b} \ltimes \gp{c_1,\ldots,c_4})\times C_2^3,
\]
in which
$(c_1,a)=(c_1,b)=(c_4,b)=1$, $(c_2,a)=(c_3,a)=c_1c_2c_3$, $(c_4,a)=c_1$, $(c_2,b)=(c_3,b)=c_2c_3$ and  $(c_4,a)=1$.
This yields that  $V_*({F}G)'\cong G'\times C_2^3 \cong C_2^4$.

\noindent
Let   $G=D^-_{16}$. According to \cite[Theorem 4]{Bovdi_Erdei_I} we have
\[
V_*({F}G)\cong G \ltimes C_2^7=\gp{a,b} \ltimes \gp{c_1,\ldots,c_7},
\]
in which $(c_1,a)=c_2c_4c_5c_7$, $(c_2,a)=(c_3,a)=1$, $(c_4,a)=(c_6,a)=c_4c_6$, $(c_5,a)=(c_7,a)=c_5c_7$,
$(c_1,b)=c_2c_3c_5c_6c_7$, $(c_2,b)=(c_3,b)=(c_4,b)=(c_6,b)=1$ and $(c_5,b)=(c_7,b)=c_5c_7$.
Moreover,  $V_*(FG)'\cong C_4 \times C_2^4$.

\noindent
Let  $G=D_8 \Y C_4$.  Then $V_*({F}G)\cong G \times C_2^7$ and  $V_*({F}G)'\cong C_2$ by  \cite[Example 5]{Bovdi_Erdei_II}.

We have proved that for each pair of the groups $\{C_4 \ltimes C_4, D^-_{16}, D_8 \Y C_4\}$ the corresponding $*$-unitary subgroups are pairwise not  isomorphic.

\noindent
$\bullet$  Let $G\in\{D_{16}, G(4,4)\} $, where $G(4,4) =\gp{a,b,c\;|\; a^4=b^2=c^2=1,(a,b)=1,(a,c)=b,(b,c)=1}$. Then  $|V_*({F}G)|=2^{12}$ by  \cite{Bovdi_Erdei_II, Bovdi_Erdei_I, Gildea_I}. Using the generating set from \cite{Bovdi_Erdei_I} and \cite[Example 6]{Bovdi_Erdei_II},  we prove that these unitary subgroups are pairwise not isomorphic.


\noindent
Let $G=G(4,4)$. According to \cite[Example 7]{Bovdi_Erdei_II}, we have
\[
V_*({F}G)\cong (G \ltimes C_2^6) \times C_2^2=\gp{a,b,c}\ltimes \gp{d_1,\ldots,d_6} \times C_2^2
\]
in which
$(d_1,a)=(d_2,a)=1$, $(d_3,a)=(d_4,a)=d_3d_4$, $(d_5,a)=(d_6,a)=d_5d_6$,
$(d_1,c)=(d_2,c)=d_1d_2$, $(d_3,c)=(d_4,c)=1$ and  $(d_5,c)=(d_6,c)=d_5d_6$.
Since $b$ is a central element,  $V_*({F}G)'\cong G'\times C_2^3 \cong C_2^4$.

\noindent
Let  $G=D_{16}$.  According to  \cite[Theorem 2]{Bovdi_Erdei_I}, we have
\[
V_*({F}G)\cong G \ltimes (D_8 \times C_2^5)=G \ltimes (\gp{c_1,c_2}\times \gp{c_3,\ldots,c_7} ),
\]
in which  $(c_1,a)=c_6c_7$, $(c_2,a)=(c_1,b)=(c_2,b)=c_3c_5c_6c_7$, $(c_3,a)=1$, $(c_4,a)=(c_6,a)=c_4c_6$, $(c_5,a)=(c_7,a)=(c_5,b)=(c_7,b)=c_5c_7$ and $(c_3,b)=(c_4,b)=(c_6,b)=1$. We have proved that $V_*({F}G)'\cong G'\times D_8'\times C_2^4 \cong C_4 \times C_2^5$.
Therefore the corresponding unitary subgroups are not isomorphic groups.


\noindent
$\bullet$ Finally, let   $G=D_8 \times C_2$.  According to   \cite[Example 7]{Bovdi_Erdei_II}, we have  $V_*({F}G)\cong (G \ltimes C_2^6) \times C_2^3$, so  $|V_*({F}G)|=2^{13}$ which completes the proof.  \end{proof}

Note that Theorem \ref{T:3} was  verified  by   GAP   package  RAMEGA  \cite{ramega}. Also,
the use of the computational algebra system  GAP \cite{GAP4} itself was very successful when we  started our investigation.

\noindent
{\bf Acknowledgment.} We would like to express our deep gratitude to both of  referees  for
the thoughtful and constructive review of our manuscript.

\end{document}